\documentclass[11pt]{amsart}
\usepackage{amscd}
\usepackage{amssymb}
\usepackage[margin=1in,includeheadfoot]{geometry}
\usepackage{cite}
\usepackage [latin1]{inputenc}
\usepackage{tabularx,booktabs,tikz}
\usepackage{caption}
\usepackage{amsmath}
\usepackage{amsfonts}

\usepackage{amscd}
\usepackage{amsthm}
\usepackage{amssymb} \usepackage{latexsym}
\usepackage{eufrak}
\usepackage{euscript}
\usepackage{epsfig}
\usepackage{graphics}
\usepackage{array}
\usepackage{enumerate}
\usepackage{dsfont}
\usepackage{color}
\usepackage{wasysym}
\usepackage{hyperref}
\usepackage{pdfsync}

\def\beq{\begin{equation}}
\def\eeq{\end{equation}}
\def\ba{\begin{array}}
\def\ea{\end{array}}

\def\R{\mathbb R}

\def\N{\mathbb N}

\newcommand{\T}{{\mathcal T}}

\newtheorem{thm}{Theorem}[section]
\newtheorem{lm}[thm]{Lemma}

\newtheorem{crl}[thm]{Corollary}

\theoremstyle{definition}
\newtheorem{rem}[thm]{Remark}

\theoremstyle{remark}

\begin{document}
\pagestyle{plain}
\title{The ground state solutions of nonlinear Schr\"{o}dinger equations with Hardy weights on lattice graphs}

\author{Lidan Wang}
\email{wanglidan@fudan.edu.cn}
\address{Lidan Wang: School of Mathematical Sciences, Fudan University, Shanghai 200433, China}


\begin{abstract}
In this paper, we study the nonlinear Schr\"{o}dinger equation
$$
-\Delta u+(V(x)- \frac{\rho}{(|x|^2+1)})u=f(x,u)
$$
on the lattice graph $\mathbb{Z}^N$ with $N\geq 3$, where $V$ is a bounded periodic potential and $0$
lies in a spectral gap of the Schr\"{o}dinger operator $-\Delta+V$. Under some assumptions on the nonlinearity $f$,  we prove the existence and asymptotic  behavior of ground state solutions with small $\rho\geq 0$ by the generalized linking theorem.

\end{abstract}

 \maketitle



\section{Introduction}
The nonlinear Schr\"{o}dinger equation 
\begin{eqnarray*}\label{0}
-\Delta u+(V(x)-\frac{\rho}{|x|^2})u=f(x,u),\quad x\in\mathbb{R}^N
\end{eqnarray*}
has drawn a great deal of interest in recent years. In particular, for $\rho=0$, there is a broad literature treating the Schr\"{o}dinger equation with periodic potential. For example, when the operator $-\Delta+V$ is positive definite, Pankov \cite{P1} proved an existence result by the Nehari variational principle and concentration
compactness methods. (Even more general asymptotically periodic case
was treated in that paper). Later, Rabinowitz \cite{RH0} obtained the
existence of nontrivial solutions under less restrictive assumptions on the
nonlinearity $f$. Moreover, in \cite{LWZ}, the authors established the ground state solutions  under a more natural super-quadratic condition (see (F4) below). When $0$ lies in a
finite spectral gap and the operator $-\Delta+V$ is not positive definite,
the first existence results (under very strong assumptions on the nonlinearity)
were found in \cite{AL,J}. Later, Troestler and Willem \cite{TW} and
Kryszewski and Szulkin \cite{KS} proved the existence of nontrivial
solutions under much more natural conditions. Pankov \cite{P} demonstrated the existence of ground state solutions by the Nehari manifold method to the case of strongly indefinite functionals. Moreover, Szulkin and Weth \cite{SW} obtained the ground state solutions based on a direct and simple reduction of the indefinite variational problem to a definite one. After that, Liu \cite{L} improved the result of Szulkin and Weth \cite{SW} under a weaker monotonicity condition on $f$. Recently, for $\rho>0$, Guo and Mederski \cite{GM} studied the existence and behavior of ground state solutions under some conditions on $f$. Later, the authors in \cite{LHT} also established the existence and asymptotical behavior of ground state solutions under different assumptions on $f$. For more related results, we refer readers to \cite{CY,CT,JT,LLT,RH,RW,WZ} and the references therein.

Nowadays, more and more researchers  turn to study  differential equations on graphs, especially for the nonlinear Schr\"{o}ldinger equations. For example, a class of Schr\"{o}dinger equations with the nonlinearity of power type have been studied on graphs, see \cite{GJ,GLY1,GLY2,HSZ,ZZ}. In addition, the existence or multiplicity of gap solitons (then the associated energy functional is strongly indefinite) of periodic discrete  Schr\"{o}ldinger equation on the lattice graph $\mathbb{Z}$ has been extensively investigated. For example,
 Pankov \cite{P2} obtained the existence of nontrivial solutions
by a generalized linking theorem due to \cite{KS}. Pankov \cite{P3} also obtained
the existence of ground state solutions by a generalized Nehari manifold and
periodic approximation technique. Later, Chen and Ma \cite{CM2} proved the existence of ground state solitons and
the existence of infinitely many pairs of geometrically distinct solitons by the generalized Nehari manifold method
developed by Szulkin and Weth \cite{SW}. Moreover, Chen and Ma \cite{CM1,CM3} established the existence of nontrivial solutions with
asymptotically or super linear terms by a variant generalized weak linking
theorem.  For related works, we refer readers to \cite{LZY,S,SM,SZ,YCD}.
 
As far as we know, there is no existence results for the Schr\"{o}dinger equation with hardy potential on the lattice graph $\mathbb{Z}^N$, which is a natural discrete model for the Euclidean space.  Motivated by the works mentioned above, in this paper, we prove the existence and asymptotical behavior of ground state solutions for a class of strongly indefinite problems with hardy weights on $\mathbb{Z}^N$ with $N\geq3$ by following the arguments in \cite{GM,M}.

Let $\Omega$ be a subset of $\mathbb{Z}^N$, we denote by $C(\Omega)$ the space of real-valued functions on $\Omega$. The support of $u\in C(\Omega)$ is
defined as $\text{supp}(u):=\{x \in \Omega : u(x)\neq 0\}$. Moreover, we denote by the $\ell^p(\Omega)$ the space of $\ell^p-$summable functions on $\Omega$. For convenience, for any $u\in C(\Omega)$, we always write
$
\int_{\Omega}u\,d\mu:=\sum\limits_{x\in \Omega}u(x),$ where $\mu$ is the counting measure in $\Omega$.

In this paper, we study the nonlinear Schr\"{o}dinger equation
\begin{equation}\label{1.1}
-\Delta u+(V(x)- \frac{\rho}{(|x|^2+1)})u=f(x,u)
\end{equation}
on the lattice graph $\mathbb{Z}^N$ with $N\geq 3$.  Here the operator $\Delta$ is the discrete Laplacian defined as $\Delta u(x)=\underset {y\sim x}{\sum}(u(y)-u(x))$.
We always assume that 
\begin{enumerate}
\item[(H):] $V\in L^{\infty}(\mathbb{Z}^N)$, $V$ is $T$-periodic with $T\in \mathbb{Z}^N$ and $$\sigma^-:=\sup~[\sigma(-\Delta+V)\cap(-\infty,0)]<0<\sigma^+:=\inf~[\sigma(-\Delta+V)\cap(0,+\infty)],$$
where $\sigma(-\Delta+V)$ is the spectrum of the operator $-\Delta+V$ in $\ell^2(\mathbb{Z}^N).$
\item[(F1):] $f:\mathbb{Z}^N\times \R\rightarrow \R$ is $T$-periodic in $x$ and continuous in $u\in\R$;

\item[(F2):] There are constants $a>0$ and $p>2$ such that
$$
|f(x,u)|\leq a(1+|u|^{p-1}), \quad (x,u)\in \mathbb{Z}^N\times \R;
$$
\item[(F3):]  $f(x,u)=o(u)$ uniformly in $x$ as $|u|\rightarrow 0$;
\item[(F4):] $\frac{F(x,u)}{u^2}\rightarrow+\infty$ uniformly in $x$ as $|u|\rightarrow +\infty$ with $F(x,u)=\int_{0}^{u}f(x,t)\,dt$;
\item[(F5):] $u\mapsto \frac{f(x,u)}{|u|}$ is non-decreasing  on $(-\infty, 0)$ and $(0, +\infty)$;
\item[(F6):] $f(x,u)$ is of $C^1$ class about $u\in\R$ and satisfies
$$
f(x,u)u-2F(x,u)\geq b|u|^q, \quad (x,u)\in \mathbb{Z}^N\times \R,
$$
where $b>0$ and $2<q\leq p$.
\end{enumerate}
Clearly, by (F1), (F2) and (F3), for any $\varepsilon>0$, there exists $c_\varepsilon>0$ such that
\begin{equation}\label{nn}
|f(x,u)|\leq\varepsilon|u|+c_\varepsilon|u|^{p-1},\quad (x,u)\in \mathbb{Z}^N\times \R.
\end{equation}
Moreover, by (F3) and (F5), we have that
\begin{equation}\label{30}
f(x,u)u\geq 2 F(x, u)\geq 0, \quad (x,u)\in \mathbb{Z}^N\times \R.
\end{equation}

Denote $A:=-\Delta+V$ and $X:=\ell^2(\mathbb{Z}^N).$ Then the energy functional of (\ref{1.1}) is
\begin{eqnarray*}
J_\rho(u)=\frac{1}{2}(Au,u)_2-\frac{1}{2}\int_{\mathbb{Z}^N}\frac{\rho}{(|x|^2+1)}|u|^2\,d\mu-\int_{\mathbb{Z}^N}F(x,u)\,d\mu,
\end{eqnarray*}
where $(\cdot,\cdot)_2$ is the inner product in $X$.
The corresponding norm in $X$ is denoted by $\|\cdot\|_2$. Then $J_\rho(u)\in C^1(X,\R)$ and the Gateaux derivative is given by
\begin{eqnarray*}
\langle J'_\rho(u),\phi\rangle=(Au,\phi)_2-\int_{\mathbb{Z}^N}\frac{\rho}{(|x|^2+1)}u\phi\,d\mu-\int_{\mathbb{Z}^N}f(x,u)\phi\,d\mu,\quad u,\phi\in X.
\end{eqnarray*}

By (H), we have the decomposition $X=X^+\oplus X^-$, where $X^+$ and $ X^-$ are the positive and negative spectral subspaces of $A$ in $X$.  Then we have that
\begin{equation*}
(Au,u)_2\geq \sigma^+\|u\|^2_2,\quad u\in X^+,\qquad \text{and}\qquad
-(Au,u)_2\geq -\sigma^-\|u\|^2_2,\quad u\in X^-.
\end{equation*}
Hence the
form $(Au,u)_2$ is positive definite on $X^+$ and negative definite on $X^-$.

For any $u,v\in X=X^+\oplus X^-$, $u=u^++u^-$ and $v=v^++v^-$, we define an equivalent inner product $(\cdot, \cdot)$ and the corresponding norm $\|\cdot\|$ on $X$ by 
$$
(u, v)=(Au^+,v^+)_2-(Au^-,v^-)_2\qquad \text{and}\qquad
\| u \|=(u, u)^{\frac{1}{2}},
$$
respectively. Clearly, the decomposition $X=X^+\oplus X^-$ is orthogonal with respect to both inner products 
 $(\cdot,\cdot)$ and $(\cdot,\cdot)_2$. Therefore, the energy functional $J_\rho$ and the corresponding Gateaux derivative can be rewritten as
\begin{eqnarray*}
J_\rho(u)=\frac{1}{2}\|u^+\|^2-\frac{1}{2}\|u^-\|^2-\frac{1}{2}\int_{\mathbb{Z}^N}\frac{\rho}{(|x|^2+1)}|u|^2\,d\mu-\int_{\mathbb{Z}^N}F(x,u)\,d\mu.
\end{eqnarray*}
and
\begin{eqnarray*}
\langle J'_\rho(u),\phi\rangle=(u^+,\phi)-(u^-,\phi)-\int_{\mathbb{Z}^N}\frac{\rho}{(|x|^2+1)}u\phi\,d\mu-\int_{\mathbb{Z}^N}f(x,u)\phi\,d\mu,\quad u,\phi\in X,
\end{eqnarray*}
respectively.

We say that $u\in X$ is a solution of (\ref{1.1}), if $u$ is a critical point of the energy functional $J_{\rho}$, i.e. $J'_{\rho}(u)=0$.
A ground state solution of (\ref{1.1}) means that
$u$ is a nontrivial critical point of $J_\rho$ with the least energy, that is,
$$
J_\rho(u)=\inf\limits_{N_\rho} J_\rho>0,
$$
where $$N_{\rho}=\{  u\in X \backslash X^-: \langle J'_\rho(u),u\rangle=0 \text{ and } \langle J'_\rho(u),v\rangle=0  \text{ for}\,  v\in  X^-\}$$
is the Nehari manifold.



Denote
\begin{equation}\label{73}
\rho^+:=\sup\,\{M>0:  (Au,u)_2\geq  M\int_{\mathbb{Z}^N}|\nabla u|^2 \,d\mu,\quad u\in X^+\}.
\end{equation}
Since $|\nabla u(x)|^2=\frac{1}{2}\underset {y\sim x}{\sum}(u(y)-u(x))^{2}$, by an elementary inequality $(a+b)^{2}\leq 2(a^2+b^2)$, one gets easily that$$\int_{\mathbb{Z}^N}|\nabla u|^2\,d\mu\leq C_N\| u \|^2_2.$$ Note that for $u\in X^+$, $(Au,u)_{2}$ is positive definite, then 
$\rho^+>0$. Let $\tilde{\rho}^+=\min \{  \rho^+, 1  \}$ and $\kappa>0$ be the constant in Lemma \ref{04} below. Now we state our first main result of this paper.

\begin{thm}\label{thm-1.1}

Let $0\leq \rho <  \frac{\tilde{\rho}^+}{\kappa}$. Assume that (H) and (F1)-(F5) hold.
Then the equation (\ref{1.1}) has a ground state solution.
\end{thm}

The second main result is about the behavior of ground state solution in the limit $\rho\rightarrow 0^+$.

\begin{thm}\label{thm-1.2}
Let $0\leq\rho< \frac{\tilde{\rho}^+}{\kappa}$. Assume that (H) and (F1)-(F6) hold. Let $u_{\rho}$ and $u_0$ be the ground state solutions of $J_\rho$ and $J_0$. Then for $\rho_n\rightarrow 0^+$, there exists a sequence $\{ x_n\}\subset \mathbb{Z}^N$
such that $u_{\rho_n}(x+x_n)$ tends to a ground state solution $u_0$ of $J_0$ as $n\rightarrow+\infty$.
 \end{thm}

\begin{rem}\label{1}
On Euclidean spaces, the exponent $p$ appeared in the growth condition $|f(x,u)|\leq a(1+|u|^{p-1})$ is always $2<p<2^*:=\frac{2N}{N-2}$ with $N\geq3$, see for examples \cite{GM,KS,M,P,SW}. While we just assume that $2<p<+\infty$ in (F2) since we have the embedding $\ell^{s}$ into $\ell^{t}$ for $s<t$ in the discrete setting. 
\end{rem}



This paper is organized as follows. In Section 2, we present some preliminaries
including settings for graphs and some auxiliary lemmas. In Section 3, we state         a generalized linking theorem and demonstrate the functional $J_\rho\in C^1(X,\R)$ satisfies the conditions of the linking theorem. In Section 4, we study the behavior of Cerami sequences. In Section 5, we are devoted to prove Theorem \ref{thm-1.1} and Theorem \ref{thm-1.2}.

\section{Preliminaries}
In this section, we introduce some settings for graphs and give some useful lemmas.

Let $G=(\mathbb{V},\mathbb{E})$ be a connected, locally finite graph, where $\mathbb{V}$ denotes the vertex set and $\mathbb{E}$ denotes the edge set. We call vertices $x$ and $y$ neighbors, denoted by $x\sim y$, if there is an edge connecting them, i.e. $(x,y)\in \mathbb{E}$.
For any $x,y\in \mathbb{V}$, the distance $d(x,y)$ is defined as the minimum number of edges connecting $x$ and $y$, i.e.
$$d(x,y)=\inf\{k:x=x_0\sim\cdots\sim x_k=y\}.$$
Let $B_{r}(a)=\{x\in\mathbb{V}: d(x,a)\leq r\}$ be the closed ball of radius $r$ centered at $a\in \mathbb{V}$ and denote $|B_{r}(a)|=\sharp B^{S}_{r}(a)$ as the volume (i.e.\,cardinality) of the set $B_{r}(a)$. For brevity, we write $B_{r}:=B_{r}(0)$.

In this paper, we consider the natural discrete model of the Euclidean space, the integer lattice graph. The $N$-dimensional integer lattice graph, denoted by $\mathbb{Z}^N$, consists of the set of vertices $\mathbb{V}=\mathbb{Z}^N$ and the set of edges $\mathbb{E}=\{(x,y): x,\,y\in\mathbb{Z}^{N},\,\underset {{i=1}}{\overset{N}{\sum}}d(x_{i},y_{i})=1\}.$
In the sequel, we write the distance $d(x,y)$, as defined in the Euclidean space, as $|x-y|$ on $\mathbb{Z}^{N}$.
    
 We denote the space of real-valued functions on $\mathbb{V}$ by $C(\mathbb{V})$, and denote the subspace of functions with finite support by $C_c(\mathbb{V})$. For any $\Omega\subset \mathbb{V}$, via continuation by zero, the spaces $C(\Omega)$ and $C_c(\Omega)$
are considered to be subspaces of $C(\mathbb{V})$ and $C_c(\mathbb{V})$. 
For any $u\in C(\Omega)$, the $\ell^p(\Omega)$ space is given by
$$\ell^p(\Omega)=\{u\in C(\Omega):\|u\|_{\ell^p(\Omega)}<+\infty\},\qquad p\in[1,+\infty],$$
 where 
 $$\|u\|_{\ell^\infty(\Omega)}=\underset {x\in \Omega}{\sup}|u(x)|\quad \text{and}\quad\|u\|_{\ell^p(\Omega)}=(\sum\limits_{x\in \Omega}|u(x)|^p)^{\frac{1}{p}},\quad p\in[1,+\infty).$$
We shall write $\|u\|_p$ instead of $\|u\|_{\ell^p(\mathbb{V})}$ if $\Omega=\mathbb{V}$.

For $u,v\in C(\mathbb{V})$, the gradient form $\Gamma,$ called the ``carr\'e du champ" operator, is defined as
\begin{eqnarray*}
\Gamma(u,v)(x)=\frac{1}{2}\underset {y\sim x}{\sum}(u(y)-u(x))(v(y)-v(x))=: \nabla u \nabla v.
\end{eqnarray*}
In particular, we write $\Gamma(u)=\Gamma(u,u)$ and denote the length of $\Gamma(u)$ by
\begin{eqnarray*}
|\nabla u|(x)=\sqrt{\Gamma(u)(x)}=(\frac{1}{2}\underset {y\sim x}{\sum}(u(y)-u(x))^{2})^{\frac{1}{2}}.
\end{eqnarray*}
The Laplacian of $u$ at $x\in \mathbb{V}$ is defined as $\Delta u(x)=\underset {y\sim x}{\sum}(u(y)-u(x)).$
For convenience, for any $u\in C(\Omega)$, we always write
$
\int_{\Omega}u\,d\mu:=\sum\limits_{x\in \Omega}u(x),$ where $\mu$ is the counting measure in $\Omega\subset\mathbb{V}$.

Next, we give some useful lemmas. First, we recall a variant of Hardy type inequality, see \cite{RS1}.
\begin{lm}\label{04}
Let $N\geq3$. We have the discrete Hardy inequality
\begin{equation}\label{ff}
\int_{\mathbb{Z}^N}\frac{|u|^2}{(|x|^2+1)}\,d\mu\leq \kappa\int_{\mathbb{Z}^N}|\nabla u|^2\,d\mu,\quad u\in C_c(\mathbb{Z}^N),
\end{equation}
where $\kappa$ depends only on $N$.
\end{lm}

\begin{lm}\label{lm-2.4}
\rm{
For any $\varepsilon>0$, there exists $C_\varepsilon>0$ such that for any $u\in X$,
$$
\int_{\mathbb{V}}F(x,u)\,d\mu\leq \varepsilon \|u\|^2_2+C_\varepsilon   \|u\|^p_p.
$$
}
\end{lm}
\begin{proof}
It follows from (\ref{nn}) and (\ref{30}) that
\begin{eqnarray*}
\int_{\mathbb{V}}F(x,u)\,d\mu &\leq &
\frac{1}{2}\int_{\mathbb{V}}f(x,u)u \,d\mu\\
&\leq& \frac{1}{2}(\int_{\mathbb{V}}\varepsilon|u|^{2}+c_{\varepsilon}|u|^p\,d\mu)\\
& \leq & \varepsilon \|u\|^2_2+C_\varepsilon   \|u\|^p_p.
\end{eqnarray*}
\end{proof}

\begin{lm}\label{lm-3.4}
Let $0\leq \rho < \frac{\tilde{\rho}^+}{\kappa}$. For any $u\in X^+$, $\| u \|^{2}_{\rho}:=(\|u\|^{2}-\int_{\mathbb{V}}\frac{\rho}{(|x|^2+1)}|u|^2\,d\mu)$
satisfies
\begin{eqnarray*}\label{2.32}
\| u \|^2\geq \| u \|_{\rho}^2\geq \frac{1}{2}  (\tilde{\rho}^+-  \kappa\rho  ) \| u \|^2.
\end{eqnarray*}
Hence $\|\cdot\|_{\rho}$ is a norm defined on $X^+$ and it is equivalent with the norm $\|\cdot\|.$
\end{lm}
\begin{proof}
  We only need to prove that $\| u \|_{\rho}^2\geq \frac{1}{2}  (\tilde{\rho}^+-  \kappa\rho  ) \| u \|^2.$
In fact, for any $u\in X^+$, we have that $\|u\|^2=(Au,u)_{2}$. Then by the Hardy inequality (\ref{ff}), we get that
\begin{eqnarray*}
\int_{\mathbb{V}}|\nabla u|^2- \frac{\rho}{(|x|^2+1)}|u|^2 \,d\mu\geq(1-\kappa\rho)\int_{\mathbb{V}}|\nabla u|^2\,d\mu.
\end{eqnarray*}
This imlies that
\begin{eqnarray*}\label{f1}
\|u\|_{\rho}^2\geq \int_{\mathbb{V}}V(x) |u|^2 \,d\mu\geq \frac{1}{2}  (\tilde{\rho}^+-\kappa\rho)\int_{\mathbb{V}}V(x) |u|^2 \,d\mu.
\end{eqnarray*}
Moreover, by (\ref{73}) and the Hardy inequality (\ref{ff}), one has that
\begin{eqnarray*}
\|u\|_{\rho}^2\geq
(\tilde{\rho}^+-  \kappa\rho  )\int_{\mathbb{V}}|\nabla u|^2\,d\mu.
\end{eqnarray*}
Hence the result $\| u \|_{\rho}^2\geq \frac{1}{2}  (\tilde{\rho}^+-  \kappa\rho  ) \| u \|^2$ follows from the last two inequalities. 
\end{proof}

\begin{lm}\label{lm-3.1}

If $\underset{n\rightarrow+\infty}{\lim}|x_n|=+\infty$, then for any $u\in X$, as $n\rightarrow+\infty,$
\begin{equation*}
\int_{\mathbb{V}}    \frac{1}{(|x|^2+1)}|u(x-x_n)|^2 \,d\mu\rightarrow0.
\end{equation*}

\end{lm}
\begin{proof}
Let $\phi_m\in C_{c}(\mathbb{V})$ and $\phi_m\rightarrow u$ in $X$ as $m\rightarrow +\infty$. Assume that $\text{supp}(\phi_m)\subset B_{r_m}$ with $r_m\geq1$.
Since $\underset{n\rightarrow+\infty}{\lim}|x_n|=+\infty$, for any $m$, there exists $n=n(m)$ such that $|x_n|-r_m\geq m$ and $\{n(m)\}$ is an increasing sequence. Then 
\begin{eqnarray*}
\int_{\mathbb{V}}    \frac{1}{(|x|^2+1)}|\phi_m(x-x_n)|^2 \,d\mu&= &
 \int_{\mathbb{V}} \frac{1}{(|x+x_n|^2+1)}|\phi_m|^2\,d\mu\\
& = & \int_{B_{r_m}} \frac{1}{(|x+x_n|^2+1)}|\phi_m|^2\,d\mu\\
& \leq & \frac{1}{(|x_n|-r_m)^2}\int_{B_{r_m}} |\phi_m|^2\,d\mu\\
& \leq & \frac{1}{m^2}\|\phi_m\|_2^2
\rightarrow 0,\qquad \text{as}~m\rightarrow+\infty.
\end{eqnarray*}
Then by the Hardy inequality (\ref{ff}), we get the result.
\end{proof}

Let $(\Omega, \Sigma, \tau)$ be a measure space, which consists of a set $\Omega$ equipped with a $\sigma-$algebra $\Sigma$ and a Borel measure $\tau:\Sigma\rightarrow[0,+\infty]$. We introduce the classical Br\'{e}zis-Lieb lemma \cite{BL}.

\begin{lm}\label{i}
(Br\'{e}zis-Lieb lemma)
Let $(\Omega, \Sigma, \tau)$ be a measure space and $\{u_n\}\subset L^{p}(\Omega, \Sigma, \tau)$ with $0<p<+\infty$. If
\begin{itemize}
\item[(a)]
$\{u_n\}$ is uniformly bounded in $L^{p}(\Omega)$,\\

\item[(b)] $u_n\rightarrow u, \tau-$almost everywhere in $\Omega$,
\end{itemize}
then we have that
\begin{eqnarray*}
\underset{n\rightarrow+\infty}{\lim}(\|u_n\|^{p}_{L^{p}(\Omega)}-\|u_n-u\|^{p}_{L^{p}(\Omega)})=\|u\|^{p}_{L^{p}(\Omega)}.
\end{eqnarray*}

\end{lm}

\begin{rem}\label{lm4}

 If $\Omega$ is countable and $\tau$ is the counting measure $\mu$ in $\Omega$, then we get a discrete version of the Br\'{e}zis-Lieb lemma.

\end{rem}

We give a discrete Lions lemma corresponding to Lions \cite{L1} on $\mathbb{R}^{N}$, which denies a sequence $\{u_n\}$ to distribute itself over $\mathbb{V}$.
\begin{lm}\label{88}

(Lions lemma)
Let $1\leq p<+\infty$. Assume that $\{u_n\}$ is bounded in $\ell^{p}(\mathbb{V})$ and $\|u_{n}\|_{\infty}\rightarrow0,$ as $n\rightarrow\infty.$
Then for any $p<q<+\infty$, as $n\rightarrow\infty$,
\begin{eqnarray*}
u_n\rightarrow0,\qquad \text{in}~\ell^{q}(\mathbb{V}).
\end{eqnarray*}

\end{lm}

\begin{proof}
For $p<q<+\infty$, this result follows from the interpolation inequality
\begin{eqnarray*}
\|u_n\|^{q}_{q}\leq\|u_n\|_{p}^{p}\|u_n\|_{\infty}^{q-p}.
\end{eqnarray*}

\end{proof}

Finally, we prove that the direct sum $X^+\oplus X^-$ in $X$ associated to a decomposition of the spectrum of the operator $A$ remains "topologically direct" in the $\ell^p(\mathbb{V})$ space.
\begin{lm}\label{jj}
Let $X^+\oplus X^-$ be the decomposition of $X=\ell^2(\mathbb{V})$ according to the positive and negative part of the spectrum $\sigma(A)$. Assume that $P,\,Q: X\rightarrow X$ are the projectors onto $X^-$ along $X^+$ and onto $X^+$ along $X^-$, respectively. Then for any $p\in[1,+\infty]$, the restrictions of $P$ and $Q$ to $X\cap \ell^p(\mathbb{V})$ are $\ell^p-$continuous.

\end{lm}

\begin{proof}
Assume that $\ell^p(\mathbb{V};\mathbb{C})=\ell^p(\mathbb{V})+i\,\ell^p(\mathbb{V})$ is the complexification of $\ell^p(\mathbb{V})$. Let $A_p$ be the operator
\begin{eqnarray*}
A_p: \ell^p(\mathbb{V};\mathbb{C})\rightarrow \ell^p(\mathbb{V};\mathbb{C}): u\mapsto -\Delta u+V(x)u
\end{eqnarray*}
with domain $D(A_p):=\{u\in \ell^p(\mathbb{V};\mathbb{C})|\, A_p u\in \ell^p(\mathbb{V};\mathbb{C})\}$.
Since the potential $V$ is bounded, it follows from \cite{BHK} that the spectrum $\sigma(A_p)\subset\mathbb{R}$ is independent of $p\in[1,+\infty]$, and moreover, for any $\lambda\not\in\sigma(A_p)=\sigma(A_2)=\sigma(A)$,
\begin{eqnarray*}
(A_p-\lambda)^{-1}=(A_2-\lambda)^{-1},\quad\text{on~}\ell^p(\mathbb{V};\mathbb{C})\cap \ell^2(\mathbb{V};\mathbb{C}).
\end{eqnarray*}
Then $0\not\in\sigma(A_p)$ and we assume that $P_p,Q_p$ are the projectors on the negative and positive eigenspaces of $A_p$. Since $\sigma(A_p)$ is bounded below, by Theorem 6.17 of \cite{K}, we can define the projector $P_p$ as follows:
\begin{eqnarray*}
P_p=\frac{1}{2\pi i}\int_{\Gamma}(A_p-\lambda)^{-1}\,d\lambda,
\end{eqnarray*}
where $\Gamma$ is a right-oriented curve around the negative part of $\sigma(A_p)$ but not crossing the spectrum. This yields that
\begin{eqnarray*}
P_p=P_2,\quad\text{on~}\ell^p(\mathbb{V};\mathbb{C})\cap \ell^2(\mathbb{V};\mathbb{C}).
\end{eqnarray*}
Then we get the desired result since $P=P_2|_{X}$ and $Q=I-P$. 
\end{proof}

\section{Generalized Linking Theorem}
 In this section, we first introduce a new topology $\mathcal{T}$ on the space $X$ so as to provide a generalized linking theorem involving the Nehari-Pankov
manifold, then we demonstrate that the functional $J_\rho\in C^1(X,\R)$ satisfies the conditions of the linking theorem.

Let $X=X^+\oplus X^-$ with $X^+\perp X^-$. For any $u\in X$, we write $u=u^++u^-$, where $u^+\in X^+$ and $u^-\in X^-$, as the direct sum decomposition.

Clearly, we have the norm topology $\|\cdot\|$ on $X$. Now we introduce a new topology $\mathcal{T}$ on $X$ which is
introduced by the norm
$$
\|u\|_{\T}=\max \{ \| u^+ \|, \sum_{k=1}^\infty \frac{1}{2^{k+1}}     |\langle u^-, e_k\rangle| \},
$$
where $\{e_k\}_{k=1}^{+\infty}$ is a complete orthonormal system in $X^-$ \cite{KS,W}. Observe that for any $u\in X$,
$$
\| u^+ \|\leq \|u\|_{\T}\leq \|u\|.
$$
The convergence of a sequence $\{u_n\}\subset X$ in $\T$ will be denoted by
$u_n\stackrel{\T}{\longrightarrow}u$.
Obviously, the new topology $\mathcal{T}$ is closely related to the topology on $X$ which is strong on $X^+$ and weak on $X^-$. More precisely, if  $\{u_n\}\subset X$ is bounded, then
\begin{equation}\label{v1}
u_n\stackrel{\T}{\longrightarrow}u\quad\Leftrightarrow\quad u^{+}_n\rightarrow u^+\quad\text{and}\quad u^{-}_n\rightharpoonup u^-.
\end{equation}

We will show that the functional $J_\rho$ satisfies the following conditions:
\begin{itemize}

\item[(A1)] For $\rho\geq 0$, $J_\rho$ is $\mathcal{T}-$upper semicontinuous, i.e. $J_\rho^{-1}([t, +\infty))$ is $\T$-closed for any $t\in \R$;\\

\item[(A2)] For $\rho\geq 0$, $J'_\rho$ is $\mathcal{T}-$to-weak$^*$ continuous, i.e. $J'_\rho(u_n)\rightharpoonup J'_\rho(u)$ as $u_n\stackrel{\T}{\longrightarrow}u_0$;\\

\item[(A3)] For $0\leq \rho< \tilde{\rho}^+$, there exists $r>0$ such that $m:=\inf\limits_{u\in X^+: \|u\|=r} J_\rho(u)>0$.\\

\item[(A4)] For $0\leq \rho< \tilde{\rho}^+$, if $u\in X\backslash X^-$, then there exists $R(u)>r$ such that
$$
\sup\limits_{\partial M(u)} J_\rho\leq J_\rho(0)=0,
$$
where $M(u)=\{  tu+v\in X| v\in   X^-,\, \| tu+v\| \leq R(u),\, t\geq 0  \}\subset \R^+ u\,\oplus X^-=\R^+ u^+\oplus X^-$ with $\R^+=[0,+\infty);$

\item[(A5)] For $\rho \geq0$, if $u\in N_\rho$, then $J_\rho(u)\geq J_\rho(tu+v)$ for $t\geq 0$ and $v\in X^-$.

\end{itemize}

Note that the conditions (A3) and (A4) imply that the functional $J_\rho$ satisfies the linking geometry. Hence, we introduce a generalized linking theorem.
For any $A\subset X,\,I\subset[0,+\infty)$ such that $0\in I$, and $h:A\times I\rightarrow X$, we collect the following assumptions:
\begin{enumerate}
\item[(h1):] $h$ is $\T$-continuous (with respect to norm $\|\cdot\|_{\mathcal{T}}$);\\

\item[(h2):] $h(u,0)=u$ for all $u\in A$;\\

\item[(h3):] $J_\rho(u)\geq J_\rho(h(u,t))$ for all $(u,t)\in A\times I$;\\

\item[(h4):] each $(u,t)\in A\times I$ has an open neighborhood $W$ in the product topology of
$(X, \T)$ and $I$ such that the set $\{   v-h(v,s):(v,s)\in W\cap( A\times I) \}$
is contained in a finite-dimensional subspace of $X$.
\end{enumerate}

Now we state the linking theorem, which can be seen in \cite{M,W}.
\begin{thm}\label{prop-2.2}

If $J_\rho\in C^1(X, \R)$ satisfies (A1)-(A4), then there exists a Cerami sequence $\{u_n\}$ at level $c_\rho$,
that is, $J_\rho(u_n)\rightarrow c$ and $(1+\| u_n \|)J_\rho'(u_n)\rightarrow 0$, where
$$
c_\rho:=\inf\limits_{u\in X\backslash X^-} \inf\limits_{h\in \Gamma(u)} \sup\limits_{u'\in M(u)} J_\rho(h(u',1))\geq m>0,
$$

$$\Gamma(u):=\{h: M(u)\times[0,1]\rightarrow X ~\text{satisfies~} (\text{h}1)-(\text{h}4)\}.$$
Suppose that (A5) holds, then $c_\rho\leq \inf\limits_{N_\rho} J_\rho$. If $c_\rho\geq J_\rho(u)$ for some critical point
$u\in X\backslash X^-$, then $c_\rho=\inf\limits_{N_\rho}J_\rho.$

\end{thm}

Now we are devoted to verify the conditions (A1)-(A5) so as to apply the linking theorem \ref{prop-2.2}. First, we show that $J_\rho$ satisfies (A1)-(A2).

\begin{lm}\label{lmaa}

Let $\rho\geq 0$.  Then
 $J_\rho$ is $\mathcal{T}-$upper semicontinuous and $J'_\rho$ is $\mathcal{T}-$to-weak$^*$ continuous.

\end{lm}

\begin{proof}
Assume that $u_n\stackrel{\T}{\longrightarrow}u$. Let $t\in \R$ such that
$$J_\rho(u_n)=\frac{1}{2}(\|u_n^+\|^2-\|u_n^-\|^2)-\frac{1}{2}\int_{\mathbb{V}}\frac{\rho}{(|x|^2+1)}|u_n|^2\,d\mu-\int_{\mathbb{V}}F(x,u_n)\,d\mu\geq t.$$
It is clear that $\|u^{+}_n\|$ is bounded; Since $\| u_n^- \|^2\leq\| u_n^+ \|^2-2t$, $\| u^-_n \|$ is bounded and hence $\| u_n \|$ is bounded. Passing to a subsequence if necessary,
\begin{eqnarray*}\label{aa}
u_n\rightharpoonup u,\quad \text{in~} X,\qquad\text{and}\qquad u_n\rightarrow u,\quad\text{ pointwise~in}\,\mathbb{V}.
\end{eqnarray*}

(i)
By (\ref{v1}), the weak lower semicontinuity of $\|\cdot\|$ and the Fatou lemma, we obtain that
$$J_\rho(u)=\frac{1}{2}(\|u^+\|^2-\|u^-\|^2)-\frac{1}{2}\int_{\mathbb{V}}\frac{\rho}{(|x|^2+1)}|u|^2\,d\mu-\int_{\mathbb{V}}F(x,u)\,d\mu\geq t.
$$

(ii) It is sufficient to show that for any $\phi\in C_{c}(\mathbb{V})$, $\underset{n\rightarrow+\infty}{\lim}\langle J_\rho'(u_n),\phi\rangle=\langle J_{\rho}'(u),\phi\rangle.$

Assume that $\text{supp}(\phi)\subset B_{r}$ with $r\geq1$. Since $B_{r+1}$ is a finite set in $\mathbb{V}$, $u_n\rightarrow u $ pointwise in $\mathbb{V}$ as $n\rightarrow+\infty$ and the assumption (F2), we get that
\begin{eqnarray*}
\langle J_\rho'(u_n),\phi\rangle-\langle J_\rho'(u),\phi\rangle&=&
\frac{1}{2}\sum\limits_{x\in B_{r+1}}\sum\limits_{y\sim x}[(u_n-u)(y)-(u_n-u)(x)](\phi(y)-\phi(x))\\&&+\sum\limits_{x\in B_{r}}(V(x)-\frac{\rho}{(|x|^2+1)})(u_{n}-u)(x)\phi(x)\\&&-\sum\limits_{x\in B_{r}}(f(x,u_n)-f(x,u))\phi(x)
\\&\rightarrow& 0, \qquad  n\rightarrow+\infty.
\end{eqnarray*}

\end{proof}

Then, for $0\leq \rho < \frac{\tilde{\rho}^+}{\kappa}$, we prove that $J_\rho$ satisfies (A3)-(A4).
\begin{lm}\label{lm-2.7}

Let $0\leq \rho < \frac{\tilde{\rho}^+}{\kappa}$. Then for any $u_0\in X\backslash X^-$, there exist $R(u_0)>r>0$ such that
\begin{eqnarray*}
m:=\inf\limits_{u\in X^+: \|u\|=r} J_\rho (u)> J_\rho (0)=0\geq \sup\limits_{\partial M(u_0)}  J_\rho(u),
\end{eqnarray*}
where $M(u_0)=\{u=tu_0+v\in X:\, v\in   X^-,\, \|u\| \leq R(u_0),\, t\geq 0  \}\subset \R^+ u_0\,\oplus X^-.$

\end{lm}

\begin{proof} For $u\in X^+$, by  Lemma \ref{lm-2.4} and Lemma \ref{lm-3.4}, we get that
\begin{eqnarray*}
J_\rho (u) &\geq &
\frac{1}{4}(\tilde{\rho}^+- \kappa\rho) \| u \|^2-\int_{\mathbb{V}}F(x, u)\,d\mu\\
& \geq & \frac{1}{4}(\tilde{\rho}^+- \kappa\rho)\| u \|^2-\varepsilon \|u\|_2^2-C_\varepsilon \|u\|_p^p.
\end{eqnarray*}
Note that $\|\cdot\|$ is equivalent to $\|\cdot\|_2$ on $X^+$ and $\|u\|_{p}\leq\|u\|_{2}$ for $p>2$. Hence, for $\varepsilon>0$ small enough, there exists $r>0$ small enough such that
$$
m:=\inf\limits_{u\in X^+: \|u\|=r} J_\rho (u)> J_\rho (0)=0.
$$

Now we prove that $\sup\limits_{\partial M(u_{0})}  J_\rho(u)\leq 0$. For $u_0\in X\backslash X^-$, since $\R^+ u_0\,\oplus X^-=\R^+ u_0^+\oplus X^-$, we may assume that $u_0\in X^+$. Arguing indirectly, assume that for some sequence $\{u_n\}\subset \R^+ u_0\,\oplus X^-$ with $\|u_n\|\rightarrow+\infty$ such that $J_\rho(u_n)>0$. Let $z_n=\frac{u_n}{\|u_n\|}=s_n u_0+z_{n}^-$, then $\|s_n u_0+z_n^-\|=1$. Passing to a subsequence, we assume that $s_n\rightarrow s,\,z_{n}^-\rightharpoonup z^-$ and $
z_{n}^-\rightarrow z^-$ pointwise in $\mathbb{V}$. Hence,
\begin{eqnarray}\label{05}
0&<&\frac{J_\rho (u_n)}{\|u_n\|^2}\nonumber\\&=&\frac{1}{2}(s_n^2\|u_0\|^2-\|z_n^-\|^2-\int_\mathbb{V}\frac{\rho}{(|x|^2+1)}|z_n|^2 \,d\mu)-
\int_{\mathbb{V}}\frac{F(x,u_n)}{|u_n|^{2}}z_n^{2}\,d\mu\nonumber\\&\leq&
\frac{1}{2}(s_n^2\|u_0\|^2-\|z_n^-\|^2)-\int_{\mathbb{V}}\frac{F(x,u_n)}{|u_n|^{2}}z_n^{2}\,d\mu.
\end{eqnarray}

If $s=0$, then it follows from (\ref{05}) that
\begin{eqnarray*}
0\leq\frac{1}{2}\|z_n^-\|^2+\int_{\mathbb{V}}\frac{F(x,u_n)}{|u_n|^{2}}z_n^{2}\,d\mu\leq\frac{1}{2}s_n^2\|u_0\|^2\rightarrow 0,
\end{eqnarray*}
which yields that $\|z_n^-\|\rightarrow 0$, and hence $1=\|s_n u_0+z_n^-\|^2\rightarrow 0$.
This is a contradiction.

If $s\neq 0$, since $\|u_n\|\rightarrow+\infty$, by (\ref{05}) and (F4), we get that
\begin{eqnarray*}
0&\leq&\underset{n\rightarrow+\infty}{\lim\sup}~[\frac{1}{2}s_n^2\|u_0\|^2-\int_{\mathbb{V}}\frac{F(x,u_n)}{|u_n|^{2}}z_n^{2}\,d\mu]\\&\leq&	\frac{1}{2}s^2\|u_0\|^2-\underset{n\rightarrow+\infty}{\lim\inf}~\int_{\mathbb{V}}\frac{F(x,u_n)}{|u_n|^{2}}z_n^{2}\,d\mu\\&\leq&\frac{1}{2}s^2\|u_0\|^2-\int_{\mathbb{V}}\underset{n\rightarrow+\infty}{\lim\inf}~\frac{F(x,u_n)}{|u_n|^{2}}z_n^{2}\,d\mu\\&\rightarrow& -\infty.
\end{eqnarray*}
This is impossible. Hence we complete the proof.

\end{proof}

Lemma \ref{lm-2.7} implies that the Nehari Manifold $N_\rho\neq\emptyset$.
\begin{crl}\label{lm-5.0}

If $0\leq\rho < \frac{\tilde{\rho}^+}{\kappa}$, then for any $u_0\in X\backslash X^-$, there exist $t>0$ and $v\in X^-$ such that $tu_0+v\in N_\rho$.

\end{crl}
\begin{proof}
For $u_0\in X\backslash X^-$, since $\R^+ u_0\,\oplus X^-=\R^+ u_0^+\oplus X^-$, we may assume that $u_0\in X^+$, then $tu_0\in X^+$.
Consider a map $\xi: \R^+\times X^-\rightarrow\R$ in the form
$$\xi(t,v)=-J_\rho(t u_0+v),$$
where $$J_\rho(u)=\frac{1}{2}\|u^+\|^2-\frac{1}{2}\|u^-\|^2-\frac{1}{2}\int_{\mathbb{V}}\frac{\rho}{(|x|^2+1)}|u|^2\,d\mu-\int_{\mathbb{V}}F(x,u)\,d\mu.$$
Observe that $\xi$ is bounded from below, coercive and weakly lower semicontinuous for $\rho\geq 0$.
Hence there exist $t\geq0$ and $v\in X^-$ such that $J_{\rho}(tu_0+v)=\sup\limits_{\R^+ u_0 \oplus X^-} J_{\rho}(u).$
By Lemma \ref{lm-2.7}, one gets that $t>0$, and hence $tu_0+v\in N_\rho$.

\end{proof}

The following lemma implies the condition (A5).
\begin{lm}\label{lm-2.8}

Let $\rho\geq0$. For any $u\in X\backslash X^-$,
\begin{eqnarray*}
J_\rho (u)\geq J_\rho(tu+v)- \langle J'_\rho(u),(\frac{t^2-1}{2}u+tv)\rangle,\qquad t\geq0,\quad v\in X^-.
\end{eqnarray*}

\end{lm}
\begin{proof} For $u\in X\backslash X^-,\,v\in X^-$ and $t\in[0,+\infty)$, we have that
$tu+v=tu^++(tu^-+v),$ where $tu^+\in X^+$ and $(tu^-+v)\in X^-$.
Direct calculation yields that
\begin{eqnarray*}
& &J_\rho(tu+v)-J_\rho (u)-\langle J'_\rho(u),(\frac{t^2-1}{2}u+tv)\rangle \\
&=&-\frac{1}{2}\|v\|^2-\frac{1}{2}
\int_{\mathbb{V}}\frac{\rho}{(|x|^2+1)}|v|^2\,d\mu+\int_{\mathbb{V}}\varphi(t,x)\,d\mu\\
\\&\leq&\int_{\mathbb{V}}\varphi(t,x)\,d\mu,
\end{eqnarray*}
where $\varphi(t,x):=(\frac{t^2-1}{2}u+tv)f(x,u)+F(x, u)-F(x, tu+v).$ We only need to prove, for any $x\in\mathbb{V}$, that
\begin{equation}\label{oa}
F(x, u)-F(x, tu+v)\leq -(\frac{t^2-1}{2}u+tv)f(x,u),\quad t\geq0, \quad v\in\R,
\end{equation}
since this implies that $\varphi(t,x)\leq 0$ for $t\geq0$ and $x\in \mathbb{V}$.

Now we prove (\ref{oa}). In fact, for any $x\in\mathbb{V}$ and $u\neq0$, the condition (F5) implies that
\begin{equation}\label{ob}
f(x,s)\geq\frac{f(x,u)}{|u|}|s|,\quad s\geq u.
\end{equation}
To show (\ref{oa}), without loss of generality, we assume that $u\leq tu+v$. Note that $$F(x, tu+v)-F(x, u)=\int_{u}^{tu+v}f(x,s)\,ds,$$ if $0< u\leq tu+v$ or $u\leq tu+v\leq0$, by (\ref{30}) and (\ref{ob}),
\begin{eqnarray*}
\int_{u}^{tu+v}f(x,s)\,ds\geq\frac{f(x,u)}{|u|}\int_{u}^{tu+v}|s|\,ds\geq(\frac{t^2-1}{2}u+tv)f(x,u);
\end{eqnarray*}
if $u<0\leq tu+v$, by (\ref{30}) and (\ref{ob}),
\begin{eqnarray*}
\int_{u}^{tu+v}f(x,s)\,ds\geq\int_{u}^{0}f(x,s)\,ds\geq\frac{f(x,u)}{|u|}\int_{u}^{0}|s|\,ds\geq(\frac{t^2-1}{2}u+tv)f(x,u).
\end{eqnarray*}
Hence (\ref{oa}) holds.   The proof is completed.
\end{proof}

\section{The behavior of Cerami sequences}
In this section, we study the behavior of Cerami sequences, which are useful in the proof of Theorem \ref{thm-1.1} and \ref{thm-1.2}.

\begin{lm}\label{lm-3.2}

Let $\{\rho_n\}\subset [0, +\infty)$, $\rho_n\leq \rho <\frac{\tilde{\rho}^+}{\kappa}$. If $\{u_n\}\subset X\setminus X^-$ satisfies 
$(1+\|u_n\|)J'_{\rho_n}(u_n)\rightarrow0$ and $J_{\rho_n}(u_n)$ is bounded from above, then $\{u_n\}$ is bounded.
In particular, any Cerami sequence of $J_{\rho}$ at level $c\geq 0$ is bounded for $0\leq \rho  <\frac{\tilde{\rho}^+}{\kappa}$.

\end{lm}
\begin{proof} Let $J_{\rho_n}(u_n)\leq M$. Suppose that $\| u_n \|\rightarrow+ \infty$ as $n\rightarrow+ \infty$.
Let $v_n:=\frac{u_n}{\|u_n\|}$, then up to a subsequence, we have that
\begin{eqnarray*}\label{3.5}
v_n\rightharpoonup v, \qquad \text{in}~X,\qquad \text{and~}\qquad
v_n\rightarrow v, \qquad \text{pointwise~in}~\mathbb{V}.
\end{eqnarray*}

We first claim that $\{v^{+}_{n}\}$ does not converge to 0 in $\ell^{q}(\mathbb{V})$ with $q>2$. In fact, by contradiction, we assume that $v_n^+\rightarrow0$ in $\ell^q(\mathbb{V})$.
Then it follows from Lemma \ref{lm-2.4} that, for any $s>0$, $$\int_{\mathbb{V}}F(x, sv_n^+)\,d\mu\rightarrow 0.$$ Moreover, by (\ref{30}) and the fact that $\langle J'_{\rho_n}(u_n),u_n\rangle\rightarrow 0$ as $n\rightarrow+\infty$, we have that
$$
\| u_n^+ \|^2-\| u_n^- \|^2\geq \langle J'_{\rho_n}(u_n),u_n\rangle,
$$
and hence
$$
2\| u_n^+ \|^2\geq \| u_n^+ \|^2+\| u_n^- \|^2+\langle J'_{\rho_n}(u_n),u_n\rangle=\| u_n \|^2+\langle J'_{\rho_n}(u_n),u_n\rangle.
$$
Since $v_n^+=\frac{u_n^+}{\|u_n\|}$,  passing to a subsequence if necessary, one gets that $\liminf\limits_{n\rightarrow\infty} \| v_n^+ \|^2=C>0$.
As a consequence, by Lemma \ref{lm-3.4} and Lemma \ref{lm-2.8},
\begin{eqnarray}\label{3.8}
M&\geq  &
 \limsup\limits_{n\rightarrow\infty} J_{\rho_n}(u_n)
 \geq  \limsup\limits_{n\rightarrow\infty} J_{\rho_n}(sv_n^+)\nonumber\\
& = & \frac{s^2}{2}\limsup\limits_{n\rightarrow\infty} \| v_n^+ \|^2_{\rho_n}\nonumber\\
& \geq & \frac{s^2}{4}  (\tilde{\rho}^+-\rho\kappa)\limsup\limits_{n\rightarrow\infty} \| v_n^+ \|^2\nonumber\\ &\geq&\frac{s^2}{4}C( \tilde{\rho}^+-\rho\kappa).
\end{eqnarray}
We obtain a contradiction since $s$ is arbitrary. We complete the claim.

Then by Lions Lemma \ref{88}, there exist a sequence $\{y_n\}\subset\mathbb{V}$ and a positive constant $c$ such that
$|v_n^{+}(y_n)|\geq c>0.$ Let $w_{n}(x)=v_n(x+y_n).$
Then for some $w\in X$,
\begin{eqnarray*}\label{3.10}
w_n\rightharpoonup w, \qquad \text{in}~X,\qquad\text{and~}\qquad
w_n\rightarrow w, \qquad \text{pointwise~in}~\mathbb{V},
\end{eqnarray*}
where $|w^{+}(0)|\geq c>0$. This implies that $w\neq0$.

Denote $\tilde{u}_{n}(x)=u_n(x+y_n),$ then $|\tilde{u}_{n}(x)|=|w_{n}(x)|\| u_n \|\rightarrow+\infty$ since $w(x)\neq 0$. It follows from (F4) that
$$
\frac{F(x, \tilde{u}_{n}(x))}{\| u_n \|^2}=\frac{F(x, \tilde{u}_{n}(x))}{|\tilde{u}_{n}(x)|^2}|w_{n}(x)|^2\rightarrow+\infty.
$$
Since $\langle J'_{\rho_n}(u_n),u_n\rangle\rightarrow 0$ as $n\rightarrow+\infty$, for $n$ large enough,
$$
\| u_n^+ \|^2-\| u_n^- \|^2-\int_{\mathbb{V}} \frac{\rho_n}{(|x|^2+1)}|u_n|^2\,d\mu\geq 0.
$$
This implies that $0\leq\frac{1}{\| u_n \|^2}\int_{\mathbb{V}} \frac{\rho_n}{(|x|^2+1)}|u_n|^2\,d\mu\leq1$ for $n$ large enough. Therefore by the periodicity of $F$ in $x\in \mathbb{V}$
and the Fatou lemma,
\begin{eqnarray*}
 0=\limsup\limits_{n\rightarrow\infty}\frac{J_{\rho_n}(u_n)}{\| u_n \|^2}&= &\limsup\limits_{n\rightarrow\infty}~[\frac{1}{2} (\| v_n^+ \|^2-\| v_n^- \|^2- \frac{1}{\| u_n \|^2}\int_{\mathbb{V}} \frac{\rho_n}{(|x|^2+1)}|u_n|^2\,d\mu )
\\
& & -\int_{\mathbb{V}}\frac{F(x,\tilde{u}_{n}(x) )}{\| u_n \|^2}\,d\mu]\\
&=&-\infty.
\end{eqnarray*}
We get a contradiction.
\end{proof}
\begin{lm}\label{lm-v}

Let $0\leq \rho <\frac{\tilde{\rho}^+}{\kappa}$. Assume that $\{u_n\}\subset X$ is a bounded Palais-Smale sequence of the functional $J_\rho$ at level $c_\rho\geq 0$, that is, $J'_{\rho}(u_n)\rightarrow0$ and $J_{\rho}(u_n)\rightarrow c_\rho$. Passing to a subsequence if necessary, there exists some $u\in X$ such that
\begin{itemize}
\item[(i)] $\underset{n\rightarrow+\infty}{\lim}J_\rho(u_n-u)=c_\rho-J_\rho(u)$;\\
\item[(ii)] $\underset{n\rightarrow+\infty}{\lim}J'_\rho(u_n-u)=0,\qquad \text{in~}X$.
\end{itemize}

\end{lm}

\begin{proof}
Since $\{u_n\}$ is bounded in $X$, we assume that for some $u\in X$,
\begin{eqnarray*}
u_n\rightharpoonup u, \qquad \text{in}~X,\qquad\text{and}\qquad
u_n\rightarrow u, \qquad \text{pointwise~in}~\mathbb{V}.
\end{eqnarray*}

(i) By the Br\'{e}zis-Lieb lemma \ref{i}, we obtain that
\begin{equation}\label{3.20}
\|u^{+}_n\|^{2}-\|u^{+}_n-u^+\|^{2}=\|\bar{u}^{+}\|^{2}+o(1),\qquad \|u^{-}_n\|^{2}-\|u^{-}_n-u^-\|^{2}=\|u^{{-}}\|^{2}+o(1).
\end{equation}
\begin{equation}\label{3.24}
\int_{\mathbb{V}}\frac{|u_{n}|^{2}}{(|x|^{2}+1)}\,d\mu-\int_{\mathbb{V}}\frac{|u_{n}-u|^{2}}{(|x|^{2}+1)}\,d\mu
=\int_{\mathbb{V}}\frac{|u|^{2}}{(|x|^{2}+1)}\,d\mu+o(1).
\end{equation}
We claim that
\begin{equation}\label{3.26}
\int_{\mathbb{V}}F(x,u_n)\,d\mu=\int_{\mathbb{V}}F(x,u_{n}-u)\,d\mu+\int_{\mathbb{V}}F(x,u)\,d\mu+o(1).
\end{equation}
In fact, direct calculation yields that
\begin{eqnarray*}
\int_{\mathbb{V}}F(x,u_n)\,d\mu-\int_{\mathbb{V}}F(x,u_n-u)\,d\mu&=&-\int_{\mathbb{V}}\int^{1}_{0}\frac{d}{d\theta}F(x,u_n-\theta u)\,d\theta \,d\mu\\&=&
\int^{1}_{0}\int_{\mathbb{V}}f(x,u_n-\theta u)u\, d\mu \,d\theta.
\end{eqnarray*}
Since $\{u_n-\theta u\}$ is bounded in $X$, by (\ref{nn}), we obtain that the sequence $\{f(x,u_n-\theta u)u\}$ is uniformly summable and tight over $\mathbb{V}$, that is, for any $\varepsilon>0$, there is a
$\delta>0$ such that, for any $\Omega\subset\mathbb{V}$ with the measure $\mu(\Omega)<\delta$,
$$
\int_{\Omega}|f(x,u_n-\theta u)u|\,d\mu<\varepsilon
$$
with any $n\in \N$; and
there exists $\Omega_0$ with $\mu(\Omega_0)<+\infty$ such that,
for any $n\in \N$,
$$
\int_{\mathbb{V}\backslash \Omega_0}|f(x,u_n-\theta u)u|\,d\mu<\varepsilon.
$$
 Note that
\begin{eqnarray*}
f(x,u_n-\theta u)u\rightarrow f(x,u-\theta u)u, \qquad \text{pointwise~in}~\mathbb{V},
\end{eqnarray*}
by the Vitali convergence theorem, we get that
$f(x,u-\theta u)u$ is summable and
$$
\int_{\mathbb{V}}f(x,u_n-\theta u)u\,d\mu\rightarrow \int_{\mathbb{V}}f(x,u-\theta u)u\,d\mu, \quad n\rightarrow+\infty.
$$
Then as $n\rightarrow+\infty$, we get that
$$
\int_{\mathbb{V}}F(x,u_n)\,d\mu-\int_{\mathbb{V}}F(x,u_n-u)\,d\mu\rightarrow\int_0^1 \int_{\mathbb{V}} f(x,u-\theta u)u\,d\mu \,d\theta=\int_{\mathbb{V}}F(x,u)\,d\mu.
$$
Hence (\ref{3.26}) holds. Therefore, by (\ref{3.20}), (\ref{3.24}) and (\ref{3.26}), one has that
\begin{eqnarray*}
J_{\rho}(u_n)=J_{\rho}(u_n-u)+J_{\rho}(u)+o(1).
\end{eqnarray*}
Note that $\underset{n\rightarrow+\infty}{\lim}J_\rho(u_n)=c_\rho$, hence 
\begin{eqnarray*}
J_\rho(u_n-u)=c_\rho-J_\rho(u)+o(1).
\end{eqnarray*}

(ii) For any $\phi\in C_c(\mathbb{V})$, assume that $\text{supp}(\phi)\subset B_{r}$, where $r$ is a positive constant. Since $B_{r+1}$ is a finite set in $\mathbb{V}$ and $u_n\rightarrow u $ pointwise in $\mathbb{V}$ as $n\rightarrow+\infty$, we get that
\begin{eqnarray*}
|\langle J_\rho'(u_n-u),\phi\rangle|&\leq&
\sum\limits_{x\in B_{r+1}}|\nabla(u_n-u)||\nabla\phi|+\sum\limits_{x\in B_{r}}|V(x)||u_{n}-u|\phi|\\&&+\sum\limits_{x\in B_{r}}\frac{\rho}{(|x|^2+1)}|u_{n}-u||\phi|\\&&+\sum\limits_{x\in B_{r}}|f(x,u_n-u)||\phi|\\&\leq& C\xi_{n}\|\phi\|,
\end{eqnarray*}
where $C$ is a constant not depending on $n$ and $\xi_{n}\rightarrow 0$ as $n\rightarrow+\infty$.
Hence
$$\underset{n\rightarrow+\infty}{\lim}\|J_\rho'(u_n-u)\|_{X}=\underset{n\rightarrow+\infty}{\lim}
\underset{\|\phi\|=1}{\sup}|\langle J_\rho'(u_n-u),\phi\rangle|=0.$$
\end{proof}

We give a decomposition of bounded Palais-Smale sequence of $J_{\rho}$ in discrete version.
\begin{lm}\label{lm-4.3}

Let $0\leq \rho <\frac{\tilde{\rho}^+}{\kappa}$. Assume that $\{u_n\}\subset X$ is a bounded Palais-Smale sequence of $J_{\rho}$ at level $c_\rho\geq 0$. Then there exist sequences $\{\bar{u}_i\}_{i=0}^{k}\subset X$ and $\{x_n^i\}_{0\leq i\leq k}\subset\mathbb{V}$ with $x_n^0=0,$
$|x_n^i|\rightarrow+\infty$, $|x_n^i-x_n^j|\rightarrow+\infty$, $i\neq j$, $i,j=1,2,\cdot\cdot\cdot,k$, such that, up to a subsequence,
\begin{enumerate}
\item[(i)] $J'_{\rho}(\bar{u}_0)=0;$\\

\item[(ii)] $J'_{0}(\bar{u}_i)=0 \text{ with } \bar{u}_i\neq 0 \text{ for } i=1, 2, \cdot\cdot\cdot, k;$\\

\item[(iii)] $u_n-\underset {{i=0}}{\overset{k}{\Sigma}}\bar{u}_i(x-x_n^i) \rightarrow 0 , \qquad\| u_n \|^2\rightarrow \underset {{i=0}}{\overset{k}{\Sigma}}\| \bar{u}_i \|^2,\qquad n\rightarrow\infty;$\\

\item[(iv)] $c_\rho=J_{\rho}(\bar{u}_0)+\underset {{i=1}}{\overset{k}{\Sigma}} J_{0}(\bar{u}_i).$
\end{enumerate}

\end{lm}

\begin{proof}
We assume that for some $\bar{u}_0\in X$,
\begin{eqnarray*}
u_n\rightharpoonup \bar{u}_0, \qquad \text{in}~X,\qquad\text{and}\qquad
u_n\rightarrow \bar{u}_0, \qquad \text{pointwise~in}~\mathbb{V}.
\end{eqnarray*}
Similar to the proof of (ii) in Lemma \ref{lmaa}, we obtain that $J'_{\rho}(\bar{u}_0)=0$ since $J'_{\rho}(u_n)\rightarrow0$ as $n\rightarrow+\infty$.

Let $v_n(x)=u_n(x)-\bar{u}_0(x).$
Then, we have that
\begin{eqnarray*}\label{3.18}
v_n\rightharpoonup 0,\qquad \text{in}~X,\qquad\text{and}\qquad
v_n\rightarrow 0,\qquad \text{poinwise~in} \mathbb{V}.
\end{eqnarray*}
By Lemma \ref{lm-v}, one has that
\begin{equation}\label{3.30}
\begin{array}{ll}
J_{\rho}(v_n)=c_\rho-J_{\rho}(\bar{u}_0)+o(1),\\
J'_{\rho}(v_n)=o(1),\quad\text{in~}X.
\end{array}
\end{equation}

For $\{v_{n}\}$, we discuss two cases:

{\bf Case 1.} $\underset{n\rightarrow+\infty}{\limsup}~\|v_{n}\|_{\infty}=0$. By the boundedness of $\{v_{n}\}$ in $X$ and Lemma \ref{88}, we have that $\|v_n\|_{t}\rightarrow 0$ as $n\rightarrow+\infty$ for $t>2$.
By Lemma \ref{jj}, for $t>2$, we have that
\begin{equation}\label{3.34}
v^{{+}}_n\rightarrow 0,\qquad v^{-}_n\rightarrow 0,\qquad \text{in}~\ell^{t}(\mathbb{V}).
\end{equation}
Since $J'_{\rho}(u_n)=o(1), J'_{\rho}(\bar{u}_0)=0$, $u_n=v_n+\bar{u}_0$ and $u_n^+=v_n^++\bar{u}_0^+$, we obtain that
\begin{eqnarray*}
o(1)&=&\langle J'_{\rho}(u_n),v^{{+}}_n\rangle\\&=&( u_n^+,v_n^+)-\rho\int_{\mathbb{V}}\frac{u_n v^{{+}}_n}{(|x|^{2}+1)}\,d\mu
-\int_{\mathbb{V}}f(x,u_n)v^{{+}}_n \,d\mu\\&=&
\|v^{+}_n\|^{2}-\rho\int_{\mathbb{V}}\frac{v_n v^{{+}}_n}{(|x|^{2}+1)}\,d\mu+\langle J'_{\rho}(\bar{u}_0),v^{{+}}_n\rangle+\int_{\mathbb{V}}f(x,\bar{u}_0)v^{{+}}_n \,d\mu-\int_{\mathbb{V}}f(x,u_n)v^{{+}}_n \,d\mu\\&=&
\|v^{+}_n\|^{2}-\rho\int_{\mathbb{V}}\frac{|v^{{+}}_n|^{2}}{(|x|^{2}+1)}\,d\mu-\rho\int_{\mathbb{V}}\frac{v^{-}_n v^{{+}}_n}{(|x|^{2}+1)}\,d\mu+\int_{\mathbb{V}}f(x,u_0)v^{{+}}_n \,d\mu-\int_{\mathbb{V}}f(x,u_n)v^{{+}}_n \,d\mu\\&\geq&\frac{1}{2}(\tilde{\rho}^+-\rho\kappa)\|v^{+}_n\|^{2}-\rho\int_{\mathbb{V}}\frac{v^{-}_n v^{{+}}_n}{(|x|^{2}+1)}\,d\mu+\int_{\mathbb{V}}f(x,\bar{u}_0)v^{{+}}_n \,d\mu-\int_{\mathbb{V}}f(x,u_n)v^{{+}}_n \,d\mu,
\end{eqnarray*}
which means that
\begin{equation}\label{3.38}
\frac{1}{2}(\tilde{\rho}^+-\rho\kappa)\|v^{+}_n\|^{2}\leq\rho\int_{\mathbb{V}}\frac{v^{-}_n v^{{+}}_n}{(|x|^{2}+1)}\,d\mu+\int_{\mathbb{V}}f(x,u_n)v^{{+}}_n \,d\mu-\int_{\mathbb{V}}f(x,\bar{u}_0)v^{{+}}_n \,d\mu+o(1).
\end{equation}
Similarly, we have that
\begin{eqnarray*}
o(1)&=&\langle J'_{\rho}(u_n),v^{-}_n\rangle\\&=&
-\|v^{-}_n\|^{2}-\rho\int_{\mathbb{V}}\frac{|v^{-}_n|^{2}}{(|x|^{2}+1)}\,d\mu-\rho\int_{\mathbb{V}}\frac{v^{-}_n v^{{+}}_n}{(|x|^{2}+1)}\,d\mu+\int_{\mathbb{V}}f(x,\bar{u}_0)v^{-}_n \,d\mu-\int_{\mathbb{V}}f(x,u_n)v^{-}_n \,d\mu\\&\leq&-\|v^{-}_n\|^{2}-\rho\int_{\mathbb{V}}\frac{v^{-}_n v^{{+}}_n}{(|x|^{2}+1)}\,d\mu+\int_{\mathbb{V}}f(x,\bar{u}_0)v^{-}_n \,d\mu-\int_{\mathbb{V}}f(x,u_n)v^{-}_n \,d\mu,
\end{eqnarray*}
which yields that
\begin{equation}\label{3.40}
\|v^{-}_n\|^{2}\leq-\rho\int_{\mathbb{V}}\frac{v^{-}_n v^{{+}}_n}{(|x|^{2}+1)}\,d\mu+\int_{\mathbb{V}}f(x,\bar{u}_0)v^{-}_n \,d\mu-\int_{\mathbb{V}}f(x,u_n)v^{-}_n \,d\mu+o(1).
\end{equation}
Then it follows from (\ref{3.34}), (\ref{3.38}) and (\ref{3.40}) that
\begin{eqnarray*}
\frac{1}{2}(\tilde{\rho}^+-\rho\kappa)\|v_n\|^{2}&\leq&\int_{\mathbb{V}}f(x,\bar{u}_0)v^{-}_n \,d\mu-\int_{\mathbb{V}}f(x,u_n)v^{-}_n \,d\mu
+\int_{\mathbb{V}}f(x,u_n)v^{{+}}_n \,d\mu\\&&-\int_{\mathbb{V}}f(x,\bar{u}_0)v^{{+}}_n \,d\mu+o(1)\\&\rightarrow&0,
\end{eqnarray*}
that is $v_n\rightarrow 0$ in $X$. Then the proof ends with $k=0$.

{\bf Case 2.} $\underset{n\rightarrow+\infty}{\liminf}~\|v_{n}\|_{\infty}=\delta>0$. Then there exists a sequence $\{x^1_{n}\}\subset\mathbb{V}$ such that
$|v_{n}(x^1_{n})|\geq\frac{\delta}{2}>0.$ Denote $u_{n,1}(x)=v_{n}(x+x^1_{n})$
and assume that, for some $\bar{u}_{1}\in X$,
\begin{eqnarray*}\label{3.44}
u_{n,1}\rightharpoonup \bar{u}_{1}, \qquad \text{in}~X,\qquad\text{and}\qquad
u_{n,1}\rightarrow \bar{u}_{1}, \qquad \text{pointwise~in}~\mathbb{V},
\end{eqnarray*}
where $|\bar{u}_{1}(0)|\geq\frac{\delta}{2}>0$. Since $v_n\rightarrow 0$ pointwise in $\mathbb{V}$, one gets easily that $d(x^1_{n},0)\rightarrow+\infty$ as $n\rightarrow+\infty$.

For any $\phi\in X$, by the H\"{o}lder inequality, Lemma \ref{lm-3.1} and (\ref{ff}), we get that
\begin{eqnarray*}
\int_{\mathbb{V}}\frac{1}{(|x|^{2}+1)}v_n(x)\phi(x-x^1_{n})\,d\mu\rightarrow0,\qquad\text{as}~n\rightarrow+\infty.
\end{eqnarray*}
Therefore, by the periodicity of $f$ in $x\in \mathbb{V}$, we obtain that
\begin{eqnarray}\label{bb}
o(1)&=&\langle J'_{\rho}(v_n),\phi(x-x^1_{n})\rangle\nonumber\\&=&( v_n^+,\phi(x-x^1_{n}))-( v_n^-,\phi(x-x^1_{n}))-\int_{\mathbb{V}}\frac{\rho}{(|x|^{2}+1)}v_n \phi(x-x^1_{n})\,d\mu
\nonumber\\&&-\int_{\mathbb{V}}f(x,v_n)\phi(x-x^1_{n}) \,d\mu\nonumber\\&=&( u_{n,1}^+,\phi)-(u_{n,1}^-,\phi)-\int_{\mathbb{V}}f(x,u_{n,1})\phi \,d\mu+o(1)\nonumber\\&=&\langle J'_{0}(u_{n,1}),\phi\rangle+o(1).
\end{eqnarray}
This means that $\langle J'_{0}(\bar{u}_{1}),\phi\rangle=0$ and $\bar{u}_1$ is a nontrivial critical point of $J_0$.

Let
\begin{equation}\label{3.46}
z_n(x)=u_n(x)-\bar{u}_0(x)-\bar{u}_1(x-x^1_{n}).
\end{equation}
Then we have that
\begin{eqnarray*}\label{3.48}
z_n\rightharpoonup 0, \qquad \text{in}~X,\qquad\text{and}\qquad
z_n\rightarrow 0, \qquad \text{pointwise~in}~\mathbb{V}.
\end{eqnarray*}
 Observe that $v_n(x)=\bar{u}_1(x-x^1_{n})+z_n(x)$, by (\ref{3.20}) and the Br\'{e}zis-Lieb lemma,
\begin{eqnarray*}\label{3.49}
\begin{array}{ll}
\|u^{{+}}_n\|^{2}=\|\bar{u}^{{+}}_0\|^{2}+\|\bar{u}^{+}_1\|^{2}+\|z^{+}_n\|^{2}+o(1),\qquad \|u^{{-}}_n\|^{2}=\|\bar{u}^{{-}}_0\|^{2}+\|\bar{u}^{-}_1\|^{2}+\|z^{{-}}_n\|^{2}+o(1).
\end{array}
\end{eqnarray*}
Then one has that
\begin{eqnarray*}
\|u_n\|^{2}=\|\bar{u}_0\|^{2}+\|\bar{u}_1\|^{2}+\|z_n\|^{2}+o(1).
\end{eqnarray*}
By (\ref{3.30}), one sees that $\{v_n\}$ is a Palais-Smale sequence of $J_{\rho}$ at level $c_\rho-J_{\rho}(\bar{u}_0)$. Then it follows from Lemma \ref{lm-v} that
\begin{eqnarray*}
\begin{array}{ll}
\underset{n\rightarrow+\infty}{\lim}J_\rho(z_n)=\underset{n\rightarrow+\infty}{\lim}(c_\rho-J_\rho(\bar{u}_0)-J_\rho(\bar{u}_1(x-x^1_{n})),\\
\underset{n\rightarrow+\infty}{\lim}J'_\rho(z_n)=0,\qquad \text{in~}X.
\end{array}
\end{eqnarray*}
Note that $\underset{n\rightarrow+\infty}{\lim}|x^1_{n}|=+\infty$, by the invariance of $\|\cdot\|$ with respect to translations, Lemma \ref{lm-3.1} and the periodicity of $F$ in $x\in \mathbb{V}$, we have that
\begin{eqnarray*}
\underset{n\rightarrow+\infty}{\lim}J_\rho(\bar{u}_1(x-x^1_{n}))&=&\underset{n\rightarrow+\infty}{\lim}[\frac{1}{2}(\|\bar{u}^+_1(x-x^1_{n})\|^2
-\|\bar{u}^-_1(x-x^1_{n})\|^2\\&&-\int_{\mathbb{V}}\frac{\rho}{(|x|^{2}+1)}|\bar{u}_1(x-x^1_{n})|^2\,d\mu)
-\int_{\mathbb{V}}F(x,\bar{u}_1(x-x^1_{n}))\,d\mu]\\&=&\frac{1}{2}(\|\bar{u}^+_1\|^2
-\|\bar{u}^-_1\|^2)-\int_{\mathbb{V}}F(x,\bar{u}_1)\,d\mu\\&=&J_0(\bar{u}_1).
\end{eqnarray*}
Hence we obtain that
\begin{eqnarray*}
\begin{array}{ll}
J_\rho(z_n)=c_\rho-J_\rho(\bar{u}_0)-J_0(\bar{u}_1)+o(1),\\
J'_\rho(z_n)=o(1),\qquad \text{in~}X.
\end{array}
\end{eqnarray*}
This implies that $\{z_n\}$ is a Palais-Smale sequence of $J_{\rho}$ at level $c_\rho-J_{\rho}(\bar{u}_0)-J_0(\bar{u}_1)$.


 For $\{z_{n}\}$,
if the vanishing case occurs for $\|z_n\|_{\infty}$, by similar arguments as in {\bf Case 1}, we obtain that $z_n\rightarrow 0$ in $X$,
and the proof ends with $k=1$.

If the non-vanishing occurs for $\|z_n\|_{\infty}$, by analogous discussions as in {\bf Case 2}, there exists a sequence $\{x_n^2\}\subset\mathbb{V}$ such that
$|z_n(x_n^2)|\geq\frac{\delta}{2}>0$. If we denote $u_{n,2}(x)=z_n(x+x_n^2)$ and assume that
\begin{eqnarray*}\label{3.58}
u_{n,2}\rightharpoonup \bar{u}_{2}, \qquad \text{in}~X,\qquad\text{and}\qquad
u_{n,2}\rightarrow \bar{u}_{2}, \qquad \text{pointwise~in}~\mathbb{V}.
\end{eqnarray*}
Then one gets that $|\bar{u}_2(0)|> 0$. This implies that $|x_n^2|\rightarrow+\infty$. In fact, we also have that $|x_n^2-x_n^1|\rightarrow+\infty$ as $n\rightarrow+\infty$. By contradiction, assume that $\{x^{2}_n-x^{1}_n\}$ is bounded in $\mathbb{V}$. Thus there exists a point $x_{0}\in\mathbb{V}$ such that $(x^{2}_n-x^{1}_n)\rightarrow x_0$ as $n\rightarrow+\infty.$ For $x_0\in \mathbb{V}$, it follows from (\ref{3.46}) that
\begin{eqnarray*}\label{7-18}
u_{n,2}(x_0+x^{1}_n-x^{2}_n)=u_{n,1}(x_0)-\bar{u}_{1}(x_0).
\end{eqnarray*}
Since $(x^{2}_n-x^{1}_n)\rightarrow x_0$ as $n\rightarrow+\infty$ and $|\bar{u}_{2}(0)|>0$, one sees that $u_{n(2)}(x_0-x^{2}_n+x^{1}_n)\not\rightarrow 0$ in the left hand side of the above equality. While the right hand side tends to zero as $n\rightarrow+\infty$. We get a contradiction.

Since $\langle J'_{\rho}(z_n),\phi(x-x^2_{n})\rangle=o(1)$, similar arguments to (\ref{bb}), we can prove that $\langle J'_{0}(\bar{u}_{2}),\phi\rangle=0$, and hence $\bar{u}_2$ is a nontrivial critical point of $J_0$.

Let
\begin{eqnarray*}\label{3.60}
w_n(x)=u_n(x)-\bar{u}_0(x)-\bar{u}_1(x-x_n^1)-\bar{u}_{2}(x-x_n^2).
\end{eqnarray*}
Then we have that
\begin{eqnarray*}\label{3.62}
w_n\rightharpoonup 0, \qquad \text{in}~X,\qquad\text{and}\qquad
w_n\rightarrow 0, \qquad \text{pointwise~in}~\mathbb{V}.
\end{eqnarray*}
Note that $z_n(x)=\bar{u}_2(x-x^2_n)+w_n(x)$, similarly, we have that
\begin{eqnarray*}
\begin{array}{ll}
\|u_n\|^{2}=\|\bar{u}_0\|^{2}+\|\bar{u}_1\|^{2}+\|\bar{u}_2\|^{2}+\|w_n\|^{2}+o(1),\\
J_{\rho}(w_n)=c_\rho-J_{\rho}(\bar{u}_0)-J_{0}(\bar{u}_1)-J_{0}(\bar{u}_2)+o(1),\\
J'_{\rho}(w_n)=o(1).
\end{array}
\end{eqnarray*}
We repeat the process again. We claim that the iterations must stop after finite steps.

We only need to prove that, for any $u\neq0$ with $J'_{0}(u)=0$,  there exist $\varepsilon_1>0$ and $\varepsilon_2>0$ such that $\|u\|\geq \varepsilon_1$ and $J_0(u)\geq\varepsilon_2$.

In fact, for $u\neq0$ satisfying $\langle J'_{0}(u),u^+\rangle=0$, by (F2), (F3) and the fact $\|u\|_{p}\leq\|u\|_{2}\leq C\|u\|$ for $p>2$, we have that for any $\varepsilon>0$, there is a constant $C_1>0$ such that
\begin{eqnarray*}
\|u^+\|^{2}=\int_{\mathbb{V}}f(x,u)u^+\,d\mu\leq\varepsilon\|u^+\|\|u\|+C_1\|u^+\|\|u\|^{p-1}.
\end{eqnarray*}
Analogously, for $\langle J'_{0}(u),u^-\rangle=0$, we get a constant $C_2>0$ such that
\begin{equation}\label{cb}
\|u^-\|^{2}=-\int_{\mathbb{V}}f(x,u)u^-\,d\mu\leq\varepsilon\|u^-\|\|u\|+C_2\|u^-\|\|u\|^{p-1}.
\end{equation}
The above two inequalities yield that
\begin{eqnarray*}
\|u\|^{2}\leq2\varepsilon\|u\|^2+2\max \{C_1, C_2\}\|u\|^{p}.
\end{eqnarray*}
Let $\varepsilon$ be small enough, then there exists $\varepsilon_1>0$  such that $\|u\|\geq \varepsilon_1$.

By Lemma \ref{lm-2.4} and (\ref{cb}), we get that
\begin{eqnarray*}
J_0(u)&=&\frac{1}{2}\|u^+\|^2-\frac{1}{2}\|u^-\|^2-\int_{\mathbb{V}}F(x,u)\,d\mu\\&=&\frac{1}{2}\|u\|^2-\|u^-\|^2-\int_{\mathbb{V}}F(x,u)\,d\mu\\&\geq&
\frac{1}{2}\|u\|^2-\varepsilon\|u\|^2-C\|u\|^{p}-\varepsilon\|u\|^2-C\|u\|^{p}\\&=&\frac{1}{2}\|u\|^2-\varepsilon\|u\|^2-C\|u\|^{p}
\end{eqnarray*}
Since $p>2$, there exists $\varepsilon_2>0$ small enough such that $J_0(u)\geq\varepsilon_2>0.$

The proof is completed.
\end{proof}

\section{Proofs of Theorem \ref{thm-1.1} and Theorem \ref{thm-1.2}}
In this section, we are devoted to prove Theorem \ref{thm-1.1} and Theorem \ref{thm-1.2}.

\
\

{\bf Proof of Theorem \ref{thm-1.1}}:
It follows from Theorem \ref{prop-2.2} and Lemma \ref{lm-3.2} that there exists a bounded Cerami sequence $\{u_n\}$ of $J_{\rho}$ at level $c_\rho>0$ in $X$. If $\rho=0$, by
Theorem \ref{prop-2.2},
 we obtain that
$$
\inf\limits_{N_0} J_{0}\geq c_0>0.
$$
Lemma \ref{lm-4.3} implies that there is a nontrivial critical point $u_0\in N_0$ of $J_{0}$ such that $J_{0}(u_0)=c_0$. Hence $u_0$ is a ground state solution of $J_{0}$, i.e. $J_{0}(u_0)=\inf\limits_{N_0} J_{0}$. In the following, we assume that $0<\rho<\frac{\tilde{\rho}^+}{\kappa}$ and consider
$$
M(u_0)=\{u=tu_0+v\in X| v\in X^-, \|u\|\leq R(u_0),t\geq 0   \}\subset\R^+ u_0^+ \oplus X^-.
$$
For $u_n=t_nu_0+v_n\in M(u_0)$, let $u_n\rightharpoonup u=t_0u_0+v_0$ in $X$.
Passing to  a subsequence if necessary, we may assume that
\begin{equation}\label{r1}
t_n\rightarrow t_0,\quad\text{in}~\R^+,\quad v_n\rightharpoonup v_{0},\quad \text{in}~ X^-,\quad
v_n\rightarrow v_0,\quad \text{pointwise\, in}\quad\mathbb{V}.
\end{equation}
Then we have that $u\in M(u_0)$, which implies that $M(u_0)$ is weakly closed. By (\ref{r1}) and the Fatou lemma, we can prove that $J_{\rho}$ is weakly upper semicontinuous. Then $J_{\rho}$
attains its maximum in $M(u_0)$, namely, there exists $t_0u_0+v_0\in M(u_0)$ such that
$$
J_{\rho}(t_0u_0+v_0)\geq J_{\rho}(w)
$$
for any $w\in M(u_0)$. By Lemma \ref{lm-2.7}, we have that $J_{\rho}(t_0u_0+v_0)>0$ and $t_0u_0+v_0\neq 0$.
Define $h(u,s)=u$ for $u\in M(u_0)$ and $s\in [0 ,1]$. Note that (h1)-(h4) in Theorem \ref{prop-2.2} are satisfied, that is $h\in \Gamma (u_0)$. Then by Theorem \ref{prop-2.2} and Lemma \ref{lm-2.8},
we have that
\begin{equation}\label{3.66}
c_0=J_{0}(u_0)\geq J_{0}(t_0u_0+v_0)> J_{\rho}(t_0u_0+v_0)=\max\limits_{u\in M(u_0)} J_{\rho}(h(u,1))\geq c_\rho.
\end{equation}
Then it follows from Lemma \ref{lm-4.3} that $k=0$ and $J_{\rho}(\bar{u}_0)=c_\rho>0$, that is $\bar{u}_0$ is a nontrivial critical point of $J_{\rho}$. By Theorem \ref{prop-2.2} again, we get that $c_{\rho}=\inf\limits_{N_\rho}J_{\rho}.$

\

\
In order to prove Theorem \ref{thm-1.2}, we first prove a crucial lemma for the relation  between $c_\rho$ and $ c_0$ as $\rho\rightarrow 0^+$.

\begin{lm}\label{vc}
 Let $0\leq\rho< \frac{\tilde{\rho}^+}{\kappa}$. Assume that (H) and (F1)-(F5) hold. If $u_{\rho}$ and $u_0$ are the ground state solutions of $J_\rho$ and $J_0$, then we have that
$$\lim\limits_{\rho\rightarrow 0^+}c_\rho=c_0.$$
 \end{lm}

\begin{proof}
Let $u_0\in N_0$ be a ground state solution of $J_0$.  By Corollary \ref{lm-5.0}, there exist $t'>0$ and $v'\in X^-$ such that $t'u_0+v'\in N_\rho$. Then it follows from Lemma \ref{lm-2.8} that
\begin{eqnarray*}
c_0 &= &
J_{0}(u_0)\geq J_{0}(t'u_0+v')=J_{\rho}(t'u_0+v')+\frac{1}{2}\int_{\mathbb{V}}\frac{\rho}{(|x|^2+1)}|t'u_0+v'|^2\,d\mu\\
& \geq & c_\rho+\frac{1}{2}\int_{\mathbb{V}}\frac{\rho}{(|x|^2+1)}|t'u_0+v'|^2\,d\mu,
\end{eqnarray*}
which implies that
\begin{equation}\label{vv}
c_0\geq c_\rho.
\end{equation}
Let $u_\rho \in N_\rho$ be a ground state solution of $J_{\rho}$. Similarly, there exist $t>0$ and $v\in X^-$
such that $tu_\rho+v\in N_0$, and hence
\begin{eqnarray}\label{4.6}
c_\rho &= &
J_{\rho}(u_\rho)\geq J_{\rho}(tu_\rho+v)=J_{0}(tu_\rho+v)-\frac{1}{2}\int_{\mathbb{V}}\frac{\rho}{(|x|^2+1)}|tu_\rho+v|^2\,d\mu\nonumber\\
& \geq & c_0-\frac{1}{2}\int_{\mathbb{V}}\frac{\rho}{(|x|^2+1)}|tu_\rho+v|^2\,d\mu.
\end{eqnarray}

We claim that as $\rho\rightarrow0^+$,
\begin{equation}\label{vx}
\int_{\mathbb{V}}\frac{\rho}{(|x|^2+1)}|t u_\rho+v|^2\,d\mu\rightarrow0.
\end{equation}
Then the result of this lemma follows from (\ref{vv}), (\ref{4.6}) and (\ref{vx}).

Now we prove (\ref{vx}). In fact, by Theorem \ref{prop-2.2} and Lemma \ref{lm-3.2}, we have that $\{u_\rho\}$ is bounded if $\rho\rightarrow 0^+$. Take any sequence $\rho_n\rightarrow 0^+$ such that
$\rho_n\leq \rho<\frac{\tilde{\rho}^+}{\kappa}$ and let $u_n=u_{\rho_n}$.

We first prove that there is a sequence $\{y_{n}\}\subset\mathbb{V}$ such that $$|u_n^+(y_{n})|>0.$$
Otherwise by Lemma \ref{88}, we get that $u_n^+\rightarrow 0$ in $\ell^t(\mathbb{V})$ for $t>2$. Since $u_n\in N_{\rho_n}$, by (\ref{nn}) and the H\"{o}lder inequality, we have that
$$
\| u_n^+ \|^2=\int_{\mathbb{V}}\frac{\rho_n}{(|x|^2+1)}u_n u_n^+\,d\mu+\int_{\mathbb{V}}f(x, u_n)u_n^+\,d\mu\rightarrow 0,\quad n\rightarrow\infty.
$$
Hence $\limsup\limits_{n\rightarrow\infty} J_{\rho_n}(u_n)\leq 0$. However, for sufficiently small $r>0$,
\begin{eqnarray*}\label{4.6.0}
J_{\rho_n}(u_n)\geq J_{\rho_n}(\frac{r}{\| u_n^+ \|}   u_n^+)\geq \inf\limits_{n\in \N} \inf\limits_{u\in X^+: \|u\|=r} J_{\rho_n}(u)>0.
\end{eqnarray*}
This yields a contradiction.

Then passing to a subsequence if necessary, there exists $u\in X$ with $u^+(0)\neq0$ such that
\begin{equation}\label{4.10.0}
u_n(x+y_n)\rightharpoonup u(x), \quad\text{ in } X, \qquad\text{and}\qquad
u_n(x+y_n)\rightarrow u(x), \quad\text{  pointwise~in } \mathbb{V}.
\end{equation}


Denote $\tilde{u}_n(x)=u_n(x+y_n)$, let $t_n\tilde{u}_n+\tilde{v}_n\in N_{0}$ with $t_n>0$, $\tilde{v}_n(x)=v_n(x+y_n)\in X^-$. By (\ref{30}), we have that
\begin{eqnarray}\label{4.12}
\| \tilde{u}_n^+ \|^2 &= &
\| \tilde{u}_n^-+\frac{\tilde{v}_n}{t_n} \|^2+\frac{1}{t_n^2}\int_{\mathbb{V}}f(x, t_n\tilde{u}_n+\tilde{v}_n)(t_n\tilde{u}_n+\tilde{v}_n)\,d\mu\nonumber\\
& \geq & \| \tilde{u}_n^-+\frac{\tilde{v}_n}{t_n} \|^2+2\int_{\mathbb{V}}  \frac{F(x, t_n(\tilde{u}_n+ \frac{\tilde{v}_n}{t_n} ))}{t_n^2}  \,d\mu,
\end{eqnarray}
which means that $\|  \tilde{u}_n^-+ \frac{\tilde{v}_n}{t_n}  \|$ is bounded. We may assume that $\tilde{u}_n^-(x)+ \frac{\tilde{v}_n(x)}{t_n} \rightarrow v(x)$
 pointwise in $\mathbb{V}$ for some $v\in X^-$. If $t_n\rightarrow+\infty$, then $|t_n\tilde{u}_n+\bar{v}_n|=t_n | \tilde{u}^+_n+(\tilde{u}_n^-+\frac{\tilde{v}_n}{t_n}) |\rightarrow+\infty$ since  $u^+(x)+v(x)\neq 0$. By the Fatou lemma and (F4), we obtain that
$$
\int_{\mathbb{V}} \frac{F(x, t_n(\tilde{u}_n+ \frac{\tilde{v}_n}{t_n} )}{t_n^2|\tilde{u}_n+\frac{\tilde{v}_n}{t_n}|^2} |\tilde{u}_n+\frac{\tilde{v}_n}{t_n}|^2\,d\mu\rightarrow+\infty,
$$
which contradicts (\ref{4.12}). Therefore $\{t_n\}$ is bounded.  As a consequence,
$\|t_n \tilde{u}_n^+\|$ and $\|t_n \tilde{u}_n^-+\tilde{v}_n\|$ are bounded. Then by the Hardy inequality (\ref{ff}),
\begin{eqnarray*}
\frac{1}{2}\int_{\mathbb{V}}\frac{\rho_n}{(|x|^2+1)}|t_n \tilde{u}_n+\tilde{v}_n|^2\,d\mu\rightarrow0~ \text{as}~n\rightarrow+\infty.
\end{eqnarray*}

\end{proof}

{\bf Proof of Theorem \ref{thm-1.2}}:
Let $\{u_n\}$ be a sequence of ground state solutions of $J_{\rho_n}$. By similar arguments as in Lemma \ref{vc}, we can find a sequence $\{x_{n}\}\subset\mathbb{V}$ such that $|u_n^+(x_{n})|>0.$  Passing to a subsequence if necessary, there exists $u\in X$ with $u^+(0)\neq0$ such that
\begin{equation}\label{fq}
u_n(x+x_n)\rightharpoonup u(x), \quad\text{ in } X, \qquad\text{and}\qquad
u_n(x+x_n)\rightarrow u(x), \quad\text{  pointwise~in } \mathbb{V}.
\end{equation}

For $(x,u)\in \mathbb{V}\times \R$, we define
$$
G(x,u)=\frac{1}{2}f(x,u)u-F(x,u).
$$
Note that $\rho_n\rightarrow 0$ as $n\rightarrow\infty.$  Hence for any $\phi\in X$, 
\begin{eqnarray*}
\langle J'_0(u_n(x+x_n)),\phi\rangle =\langle J'_{\rho_n}(u_n(x)),\phi(x-x_n)\rangle +\int_{\mathbb{V}}\frac{\rho_n}{(|x|^2+1)}u_n(x)\phi(x-x_n)\,d\mu
\rightarrow 0.
\end{eqnarray*}
Similar to the proof of (ii) in Lemma \ref{lmaa}, we get that $\langle J'_0(u_n(x+x_n)),\phi\rangle \rightarrow \langle J'_0(u),\phi\rangle$.
Then $u$ is a nontrivial critical point of $J_0$, and hence $u\in \mathcal{N}_0$. By Lemma \ref{vc} and the Fatou lemma, we have that
\begin{eqnarray}\label{4.20}
c_0 &= &
\liminf\limits_{n\rightarrow\infty} J_{\rho_n}(u_n)=\liminf\limits_{n\rightarrow\infty} (J_{\rho_n}(u_n)-\frac{1}{2}\langle J'_{\rho_n}(u_n),u_n\rangle)\nonumber\\
& = & \liminf\limits_{n\rightarrow\infty} \int_{\mathbb{V}} G(x,u_n)\,d\mu=\liminf\limits_{n\rightarrow\infty} \int_{\mathbb{V}} G(x,u_n(x+x_n))\,d\mu\nonumber\\ &\geq& \int_{\mathbb{V}} G(x,u)\,d\mu
= J_0(u)\geq c_0.
\end{eqnarray}
This implies that $u$ is a ground state solution of $J_0$.

Let us denote $w_n(x)=u_n(x+x_n)$ and observe that
\begin{eqnarray*}\label{4.22}
\int_{\mathbb{V}}G(x,w_n)-G(x,w_n-u)\,d\mu &= &
\int_{\mathbb{V}} \int_0^1 \frac{d}{dt} G(x, w_n-u+tu)\,dt\,d\mu\nonumber\\
& = & \int_0^1  \int_{\mathbb{V}}g(x, w_n-u+tu)u\,d\mu \,dt,
\end{eqnarray*}
where $g(x,s)=\frac{\partial}{\partial s}G(x,s)$ for $s\in \R$ and $x\in\mathbb{V}$. Since $\{w_n-u+tu\}$ is bounded in $X$, by (F6) and (\ref{nn}), we can prove that the family $\{g(x, w_n-u+tu)u\}$ is uniformly summable and tight over $\mathbb{V}$.
In addition, note that $g(x, w_n-u+tu)u\rightarrow g(tu)u$  pointwise in $\mathbb{V}$,
then  by the Vitali convergence theorem, we get that
$g(x, tu)u$ is summable and
$$
\int_{\mathbb{V}}g(x, w_n-u+tu)u\,d\mu\rightarrow \int_{\mathbb{V}}g(x, tu)u\,d\mu, \quad n\rightarrow+\infty.
$$
Then we have that
$$
\int_{\mathbb{V}}G(x,w_n)-G(x,w_n-u)\,d\mu\rightarrow\int_0^1 \int_{\mathbb{V}} g(x, tu)u\,d\mu\,dt=\int_{\mathbb{V}}G(x,u)\,d\mu, \quad n\rightarrow\infty.
$$
Combined with (\ref{4.20}), we get that $\lim\limits_{n\rightarrow\infty} \int_{\mathbb{V}}G(x,w_n-u)\,d\mu=0.$
By (F6), we have that $w_n\rightarrow u$ in $\ell^q(\mathbb{V})$ with $2<q\leq p$. Note that $\{w_n\}$ is bounded in $\ell^2(\mathbb{V})$, and hence in $\ell^{\infty}(\mathbb{V})$.
By interpolation inequality, for $p<t<+\infty$,
$$
\| w_n-u\|_{t}^t\leq \| w_n-u\|_{p}^p \| w_n-u\|_{\infty}^{t-p}\rightarrow 0.
$$
Hence, one has that
$w_n\rightarrow u$ in $\ell^t(\mathbb{V})$ for $2<t<+\infty$. Combined with $\rho_n\rightarrow 0$ as $n\rightarrow\infty.$, we get that
\begin{eqnarray*}
\| w_n^+ - u^+ \|^2 &= &
\langle J'_{\rho_n}(u_n),(w_n^+ - u^+)(x-x_n)\rangle-(u^+,  w_n^+ - u^+ )\nonumber\\
&  & +\int_{\mathbb{V}}\frac{\rho_n}{(|x|^2+1)}u_n( w_n^+ - u^+)(x-x_n)\,d\mu+\int_{\mathbb{V}}f(x, w_n)(w_n^+ - u^+)\,d\mu\nonumber\\
& \rightarrow&0,
\end{eqnarray*}
\begin{eqnarray*}
\| w_n^- - u^- \|^2 &= &
\langle J'_{\rho_n}(u_n),(w_n^- - u^-)(x-x_n)\rangle-(u^-,  w_n^- - u^-)\nonumber\\
&  & -\int_{\mathbb{V}}\frac{\rho_n}{(|x|^2+1)}u_n( w_n^- - u^-)(x-x_n)\,d\mu-\int_{\mathbb{V}}f(x, w_n)(w_n^- - u^-)\,d\mu\nonumber\\
& \rightarrow&0.
\end{eqnarray*}
Therefore, $\| w_n- u\|^2=\| w_n^+ - u^+ \|^2+\| w_n^- - u^- \|^2\rightarrow0$, as $n\rightarrow+\infty$, which means that $w_n\rightarrow u$ in $X$.
\qed

\
\section{Acknowledgements}
The author would like to thank Bobo Hua and Tao Zhang for helpful discussions and suggestions.

\end{document}